\documentclass[12pt]{amsart}

\usepackage[pdfauthor={Christos Katsivelos},
pdftitle={},
          pdfkeywords={Lattice counting, Quantum variance},
           pdfcreator={Pdflatex}]
{hyperref}
\hypersetup{colorlinks=true}
\usepackage{xcolor}

\usepackage{enumerate}
\usepackage{amssymb,amsmath,latexsym,amsthm}
\usepackage[english]{babel}
\usepackage{comment}
\usepackage{todonotes} 

\theoremstyle{plain}
\newtheorem{theorem}{Theorem}[section]
\newtheorem{lemma}{Lemma}[section]

\newtheorem{corollary}{Corollary}[section]

\theoremstyle{definition}
\newtheorem{conjecture}{Conjecture}[section]
\newtheorem{remark}{Remark}[section]

\numberwithin{equation}{section}
\DeclareMathOperator*{\vol}{vol}

\begin{document}
\nocite{*}


\def \g {{\gamma}}
\def \G {{\Gamma}}
\def \l {{\lambda}}
\def \a {{\alpha}}
\def \b {{\beta}}
\def \f {{\phi}}
\def \r {{|t|}}
\def \R {{\mathbb R}}
\def \H {{\mathbb H}}
\def \N {{\mathbb N}}
\def \C {{\mathbb C}}
\def \Z {{\mathbb Z}}
\def \F {{\Phi}}
\def \Q {{\mathbb Q}}
\def \e {{\epsilon }}
\def \ev {{\vec\epsilon}}
\def \ov {{\vec{0}}}
\def \GinfmodG {{\Gamma_{\!\!\infty}\!\!\setminus\!\Gamma}}
\def \GmodH {{\Gamma\backslash\H}}
\def \sl  {{\hbox{SL}_2( {\mathbb R})} }
\def \slc  { {\hbox{PSL}_2({\mathbb C})}   }
\def \psl  {{\hbox{PSL}_2( {\mathbb R})} }
\def \so  {{\hbox{SO}^{+}(1,n)} }
\def \slz  {{\hbox{SL}_2( {\mathbb Z})} }
\def \pslzi {{\hbox{PSL}_2({\mathbb Z}[i])} }
\def \pslz  {{\hbox{PSL}_2( {\mathbb Z})} }
\def \L  {{\hbox{L}^2}}
\def \GmodHthree {{\Gamma\backslash\H^3}}
\def \GmodHn {{\Gamma\backslash\H^n}}
\def \B {{\mathcal B}}

\newcommand{\norm}[1]{\left\lVert #1 \right\rVert}
\newcommand{\abs}[1]{\left\lvert #1 \right\rvert}
\newcommand{\modsym}[2]{\left \langle #1,#2 \right\rangle}
\newcommand{\inprod}[2]{\left \langle #1,#2 \right\rangle}
\newcommand{\Nz}[1]{\left\lVert #1 \right\rVert_z}
\newcommand{\tr}[1]{\operatorname{tr}\left( #1 \right)}


\title[Hyperbolic lattice counting in large dimensions]{The hyperbolic lattice counting problem in large dimensions}
\author{Christos Katsivelos}
\address{University of Patras\\
Department of Mathematics\\
26504 Patras\\
Greece}
\email{up1112463@upatras.gr}
\date{\today}


\begin{abstract}
For $n\geq 3$ and $\Gamma$ a cocompact lattice acting on the hyperbolic space $\mathbb{H}^n$, we investigate the average behaviour of the error term in the circle problem. First, we explore the local average of the error term over compact sets of $\GmodH^n$. Our upper bound depends on the quantum variance and the spectral exponential sums appearing in the study of the Prime geodesic theorem. We also prove $\Omega$-results for the mean value and the second moment of the error term.

\end{abstract}
\maketitle
\section{Introduction}

For $n \geq 2$ let $\mathbb{H}^n$ denote the $n$-dimensional hyperbolic space and let $\Gamma \subset \so$ be a discrete cofinite group of isometries acting on $\mathbb{H}^n$. The quotient space $\GmodH^n$ is a $n$-dimensional geometrically finite orbifold of constant curvature $K=-1$. Assume that $\GmodH^n$ is compact. Then the Laplace-Beltrami operator $- \Delta = -\Delta_{\Gamma} $ is a symmetric, non-negative operator on $\GmodH^n$ and has discrete spectrum in $\L(\GmodH^n)$:
\begin{eqnarray*} 
0=\lambda_0 \leq \lambda_1 \leq \lambda_2 \leq ... \leq\lambda_j \leq ...
\end{eqnarray*}
with $\lambda_j \to \infty$. 
We write $\phi_j$ for the $\L$-normalized eigenfunction of $-\Delta$ corresponding to $\lambda_j$, i.e.
\begin{eqnarray*} 
(\Delta +\lambda_j ) \phi_j(z) = 0, \quad z \in \GmodH^n. 
\end{eqnarray*}
The study of the behaviour of Laplace eigenfunctions $\f_j$ and the corresponding Laplace eigenvalues $\lambda_j$ is a central problem in Mathematical Physics and Number theory. For the rest of the paper we fix the notation $\l_j = s_j (n-1-s_j)$. Notice that $\lambda_j \geq (\frac{n-1}{2})^2$ if and only if $\Re(s_j) = \frac{n-1}{2}$. The eigenvalues satisfying $\lambda_j < (\frac{n-1}{2})^2$ are called \textit{small} (or \textit{exceptional}) eigenvalues. We also use the standard notation 
\begin{eqnarray*}
\lambda_j = \left(\frac{n-1}{2} \right)^2 + t_j^2  
\end{eqnarray*}
with $t_j \in \mathbb{R}$ or $0<|i t_j|\leq\frac{n-1}{2} $, with $t_j$ imaginary. If $\GmodH^n$ is not compact but has finite hyperbolic volume the operator $-\Delta$ has also continuous spectrum, the interval $[(n-1)^2/4, \infty)$, spanned by non-holomorphic Eisenstein series. 

For two points $z, w \in \mathbb{H}^n$ let $d(z,w)$ denote their hyperbolic distance. The hyperbolic lattice counting problem (also known as \textit{the circle problem} in $2$ dimensions) asks to estimate the growth of $\G$-orbits in $\H^n$, i.e. the counting function 
\begin{eqnarray*}
N_{\Gamma, r}(z,w) = \# \{ \gamma \in \Gamma : d (z,\gamma w) \leq r \},\, \text{as}\quad r\to\infty.
\end{eqnarray*}
 This problem can be regarded as the hyperbolic analogue of the classical lattice counting problem on the Euclidean space, which dates back to Gauss. After the change of variable $X= 2 \cosh r \asymp e^r$ we denote the quantity $N_{\Gamma, r}(z,w)$ by $N_{\Gamma}(X;z,w)$. The study of pointwise estimates for $N_{\Gamma}(X;z,w)$ reduces to estimating the spectral expansion of an appropriately defined automorphic kernel using the pre-trace formula. G\"unther \cite{gunth}, Lax and Phillips \cite{laxphillips} and Levitan \cite{levitan} (for the cocompact case), generalizing previous work of Selberg \cite{selberg}, Huber \cite{huber}, Patterson \cite{patterson}, Good \cite{good} for $n=2$ and Fricker \cite{fricker} for $n=3$, proved the following main theorem.

\begin{theorem} \label{mainformula} Assume that $\Gamma$ is a cofinite group acting on $\mathbb{H}^n$. Then the asymptotic behaviour of $N_{\Gamma}(X;z,w)$ as $X \to \infty$ is given by the formula: 
\begin{eqnarray*}
N_{\Gamma}(X;z,w) = M_{\Gamma}(X; z,w)  + E_{\Gamma}(X;z,w),
\end{eqnarray*}
where the main term $M_{\Gamma}(X; z,w)$ is given by a finite sum over the small eigenvalues $\lambda_j < (\frac{n-1}{2})^2$ and the eigenvalue $\lambda_j=\left(\frac{n-1}{2}\right)^2$; moreover the error term $E_{\Gamma}(X;z,w)$ satisfies the bound
\begin{eqnarray} \label{errorselberg}
 E_{\Gamma}(X;z,w) = O_{\Gamma} \left(X^{n-2 + \frac{2}{n+1}} (\log X)^{\frac{3}{n+1}}\right).
\end{eqnarray}
\end{theorem}
\begin{remark}
All the $O$-estimates throughout the paper will depend on the group $\Gamma$. Moreover, whenever the $O$-bound depends on $\epsilon>0$, the implied constant will also depend on $\epsilon$. For brevity reasons, from now on we will not indicate explicitly these dependencies.
\end{remark}
For the exact description  of the main term $M_{\Gamma}(X; z,w)$, which is a well understood quantity, we refer to subsection \ref{contribution of main term}. The main goal of the lattice counting problem in $\mathbb{H}^n$ is to understand the exact asymptotic behaviour of the error term $E_{\Gamma}(X;z,w)$, which contains only the contribution of the large eigenvalues $\lambda_j > (n-1)^2/4$. This problem remains far open, most famously in dimensions $2$ and $3$. 

\subsection{Local average over compact sets, quantum variance and exponential sums}

An intriguing relationship between the lattice counting problem and a quantum variance estimate was explored by Petridis-Risager for the case of the modular group $\Gamma = \pslz$ in \cite{petridisrisager}. In fact, in the $2$-dimensional case Selberg proved the upper bound  
\begin{eqnarray*}
E_{\Gamma} (X;z,w) = O(X^{2/3}),
\end{eqnarray*}
which is $X^{\epsilon}$-stronger than the upper bound (\ref{errorselberg}). 
Improving this bound amounts to detect thin cancellation in a spectral exponential sum involving the Laplace eigenfunctions $\phi_j$. Selberg's bound was recently improved for the first time in \cite{chatz3} for the modular group and for Heegner points $z,w$ of different discriminants. 

One may hope for stronger results when averaging the error term over the surface. The main result of Petridis and Risager in \cite{petridisrisager} is the following. 
\begin{theorem} 
Let $\G =\pslz$ and $f$ a smooth, compactly supported function on the modular surface $\GmodH^2$. Then
\begin{align*}
\int_{\GmodH^2} f(z) E_{\G}(X;z,z) d \mu(z)= O_{f} \left( X^{7/12+\epsilon} \right),
\end{align*}
for every $\epsilon > 0$, where $d\mu$ is the hyperbolic volume measure on $\GmodH^2$.
\end{theorem}

Recently Cherubini-Katsivelos in \cite{cherubini2} announced the following analogous conditional result for the Picard manifold. 

\begin{theorem}\label{local average cherubini}
Let $\G =\mathrm{PSL}_2(\mathbb{Z}[i])$ and
let~$f$ a smooth compactly 
supported function on the Picard manifold $\GmodH^3$.
Then, for every $\epsilon>0$, we have
\begin{equation*}
\int_{\GmodHthree} f(z) E_\G(X,z,z)\,d\mu(z)=O_{f,\theta,q}\left(X^{\frac{6-4\theta}{5-4\theta}+\epsilon}+X^{\frac{2q}{q+1}+\epsilon}\right),
\end{equation*}
where $d\mu$ is the hyperbolic volume measure on $\GmodH^3$, $\theta$ is any subconvexity exponent
for quadratic Dirichlet $L$-functions over the Gaussian integers and $q$ is an exponent subject to a conjectural statement for the Quantum variance on $\GmodH^3$. 
\end{theorem}
\begin{remark}
For more details on the exponents $\theta$ and $q$ in the statement of Theorem \ref{local average cherubini} see Hypothesis STX and Hypothesis QV in \cite{cherubini2}.
\end{remark}

An estimate of this type is called a \textit{local average} estimate over the manifold. In order to prove their result, Petridis-Risager and Cherubini-Katsivelos used strong input from the Quantum unique ergodicity problem.

Assume $\mathcal{M}$ is a compact Riemannian manifold with volume measure $d \mu$ and Laplace eigenfunctions $\phi_j$ (assumed $\L$-normalized) corresponding to eigenvalues $\lambda_j \geq 0$. Define the measures
\begin{equation}\label{measures}
    \,d \mu_j(z)=|\phi_j(z)|^2\,d \mu(z).
\end{equation}
The Quantum ergodicity (QE) theorem states that if the geodesic flow on the unit cotangent bundle is ergodic then there exists a density one subsequence $\lambda_{j_k}$ such that the measures $d \mu_{j_k}$ converge weakly to the volume measure:
\begin{equation*}
    \,d \mu_{j_k}(z) \xrightarrow{j_k\to\infty} \frac{1}{\hbox{vol}(\mathcal{M})}\,d \mu(z),
\end{equation*}
i.e.
\begin{equation}\label{overlinef}
     \int_{\mathcal{M}} f(z) \, d \mu_{j_k}(z) \xrightarrow{j_k\to\infty} \Bar{f} := \frac{1}{\hbox{vol}(\mathcal{M})}\int_{\mathcal{M}} f(z) \,d \mu(z),
\end{equation} 
for every smooth and compactly supported function $f$ on $\mathcal{M}$ (for more details, see \cite{linderstrauss, luo sarnak} and references therein). 
The Quantum unique ergodicity (QUE) conjecture of Rudnick and Sarnak \cite{rudnicksarnak} states that for compact manifolds of negative curvature the full sequence $d \mu_j$ converges to the volume measure. 
This has been proven to be the case for compact arithmetic Riemann surfaces by Linderstrauss \cite{linderstrauss} and for the cofinite case by Soundararajan \cite{soundararajan}. 
In fact, Luo and Sarnak \cite{luo sarnak} have proved the following strong quantum variance estimate.
\begin{theorem} 
    Let $\Gamma=\pslz$, $f$ a smooth compactly supported function on $\mathcal{M} = \GmodH^2$ and $\phi_j(z)$ an orthonormal basis of Hecke cusp forms. Then, for every $\epsilon>0$
    \begin{equation*}
        \sum_{|t_j|\leq T}\left|\int_{\GmodH^2}f(z)\,d\mu_j(z)-\Bar{f}\right|^2=O_{f} \left( T^{1+\epsilon}\right).
    \end{equation*}
\end{theorem}
This quantum variance estimate was later refined by Zhao \cite{Zhao} and Sarnak--Zhao \cite{Sarnak Zhao}, whereas for the holomorphic case the quantum variance was worked out by Luo--Sarnak \cite{luo sarnak 2}.
In higher dimensions the QUE conjecture remains far open. However, Zelditch \cite{zelditch1, zelditch2} proved a more general quantum variance estimate for any $n$-dimensional compact Riemannian manifold with negative curvature (not necessarily constant), but with a worse upper bound (see also \cite{Zhao, zelditch3} for more details on quantum variance estimates).

For $\GmodHn$ compact, Weyl's law (Theorem \ref{weylslaw}) implies that the size of the (necessarily discrete) spectrum of the hyperbolic Laplacian on $\Gamma \backslash\mathbb{H}^n$ is of order $\# \{j : |t_j| \leq T\} \sim c_{\Gamma} T^n.$

The following conjecture is a weak version of the quantum variance problem, see \cite[p.~2]{nelson} for further details.

\begin{conjecture}\label{weakerconjecture}
\textit{Assume $f$ is a smooth, compactly supported function on $\Gamma\backslash\H^n$, with $\Gamma$ cocompact. Then
\begin{equation}\label{quantumvariance}
    \sum_{|t_j|\leq T}\bigg|\int_{\Gamma\setminus\mathbb{H}^n}f(z)\,d\mu_j(z)-\Bar{f}\bigg|^2=O_{f}\big(T^{q_n}\big),
\end{equation}
with $q_n = q_n(\G) = 1+\epsilon$ for every $n \geq 2$.}
\end{conjecture}
Notice that any upper bound of the form $o(T^n)$ in (\ref{quantumvariance}) implies QE for the orbifold $\GmodHn$, and indeed this is what Zelditch proved. Although Conjecture \ref{weakerconjecture} is far open, we will see that a weaker version of this conjecture suffices to prove our result. 

For a compact manifold $\GmodH^n$ with eigenvalues $\lambda_j$ we define the spectral exponential sums
\begin{equation}\label{expsumsdef}
 S_{\Gamma} (T,X) = S(T,X) := \sum_{|t_j|\leq T}X^{it_j},
\end{equation}
for $X, T \geq 1$. For arithmetic Riemann surfaces this exponential sum has been studied in \cite{balkanova3,balkanova4, iwaniec 2, koyama1998, luo sarnak},  whereas for the Picard manifold the analogous sum has been investigated in \cite{balkanova, balkanova2, chatz2, kaneko 2,kaneko, koyama}. We are interested in upper bounds of the form
\begin{equation}\label{exp sums assume}
    S(T,X) \ll X^{a} T^{b},
\end{equation}
for exponents $a = a (\G),\, b = b (\G)$. Weyl's law yields the trivial estimate $a =0$, $b =n$, for every $n$. For the case of the modular surface $\pslz \backslash \H^2$ it was conjectured in \cite{petridisrisager} that (\ref{exp sums assume}) holds with $a =\epsilon,$ $b =1+\epsilon$, for any $\epsilon >0$. For the Picard manifold $\pslzi \backslash \H^3$, it is conjectured in \cite[Conjecture~5.1]{kaneko 2} that (\ref{exp sums assume}) holds with $a =\epsilon,$ $b =2+\epsilon$. It is thus justified to formulate the following conjecture, which may be well known to the experts but we have not found it anywhere explicitly stated in the literature.

\begin{conjecture} \label{exp sums}
\textit{Let $n \geq 2$ and let $\Gamma$ be an essentially cuspidal cofinite  lattice in $\so$ Then, for $X, T \to \infty$ and for every $\epsilon >0$ 
we have  
\begin{equation*}
S(T,X)=O\left(X^{\epsilon}T^{n-1+\epsilon}\right)
\end{equation*}}
\end{conjecture}

 We will only use this conjecture for the case of cocompact lattices, nevertheless it makes sense to state the conjecture for any essentially cuspidal group $\Gamma$ (see \cite{sarnak} for the definition of an essentially cuspidal group).

\subsection{Our local average result} 

Our first result is a local average for the error term $E_{\G} (X;z,w)$, with $\G$ a cocompact lattice acting on $\H^n$ (with $n\geq 3$). In this case our result is conditional on (a weaker version of) Conjecture \ref{weakerconjecture} and Conjecture \ref{exp sums}. 

\begin{theorem} \label{local average main theorem}
Let $n\geq 3$ and let $\G \subset \so $ be a cocompact group and $f$ a smooth, compactly supported function on $\GmodH^n$.
Assume Conjecture \ref{weakerconjecture} with the weaker exponent $q_n=n-2+\epsilon$  and  Conjecture \ref{exp sums}. Then
\begin{align*}
\int_{\GmodH^n} f(z) E_{\G}(X;z,z) \, d \mu(z)= O_{f} \left( X^{n-2+\epsilon} \right),
\end{align*}
where $d\mu$ is the hyperbolic volume measure on $\GmodH^n$.
\end{theorem}

It should be mentioned that any stronger assumption on the value of $q_n$ does not improve our result. Notice also that for $n=3$ our result is essentially optimal. 

\begin{remark}
Conjecture \ref{weakerconjecture} is analogous to  Hypothesis QV in \cite{cherubini2} and Conjecture \ref{exp sums} is analogous to Hypothesis STX in \cite{cherubini2}.
\end{remark}

\begin{remark}
The local average has been also studied by Biró \cite{biro} where he proved that for any cofinite group $\Gamma\subset \psl$ we can obtain the local average bound $O_{f, \epsilon} \left( X^{5/8+\epsilon} \right)$. Even though his exponent is worse than the exponent of Petridis-Risager (by a margin of $1/24$), he covers any cofinite Fuchsian group $\Gamma$ and his proof needs no arithmetic input. In particular Biró used his generalization of the Selberg trace formula from \cite{biro2}. More recently, in his recent preprint \cite{biro3}, Biró also studied a second moment local average for any cofinite Fuchsian group, obtaining a slightly worse exponent. 
\end{remark}

\begin{remark}
The lattice counting problem has also been studied in the higher rank case \cite{duke, gorodnik}. Recently, Blomer and Lutsko \cite{blomer} improved previous pointwise bounds for the group $\hbox{SL}_n(\mathbb{Z})$ and studied a version of the local average of the error term for any rank.
\end{remark}
 
\subsection{\texorpdfstring{$\Omega$}{TEXT}-results for the mean value of the error term}
One would like to expect square root cancellation for the error term:
\begin{eqnarray} \label{conjecturen}
 E_{\Gamma}(X;z,w) = O \left(X^{\frac{n-1}{2}+\epsilon} \right)
\end{eqnarray}
for any $z$ and $w$. However, Phillips and Rudnick proved that this is not always true. In \cite{phirud} they investigated the mean behaviour of the error term $E_{\Gamma}(X;z,z)$ in the radial parameter $r \asymp \log X$, as well as $\Omega$-results for the error. For $n\geq 4$, they constructed an explicit arithmetic lattice $\Gamma \subset \so \cap {\hbox{SL}_{n+1}( {\mathbb Z})}$ such that for a special fixed point $z_0$ we have
\begin{eqnarray*}
 E_{\Gamma}(X;z_0,z_0) = \Omega(X^{n-2}).
\end{eqnarray*}
Thus for large dimensions one cannot always expect square root cancellation for the error term. 
Their argument is purely arithmetic, as they embedded a Euclidean lattice counting error term in $E_{\Gamma}(X;z_0 ,z_0)$. However, for a general lattice $\Gamma \subset \so$, using the almost periodicity of the normalized error term in the $r$-variable  they only proved the following much weaker $\Omega$-result.

\begin{theorem} 
If $\Gamma$ is cocompact or congruence group, then
\begin{equation*} 
E_{\Gamma} (X;z,z) = \Omega \left(X^{\frac{n-1}{2}} (\log \log X)^{\frac{1}{2}-\frac{1}{2n}-\epsilon} \right)
\end{equation*}
for all $\epsilon>0$.
\end{theorem}
For a cofinite group $\Gamma$ they also prove the existence of the mean value in the $r$-variable:
\begin{equation} \label{mnrt}
\lim_{R \to \infty} \frac{1}{R} \int_{0}^{R}  \frac{E_{\Gamma} (e^r;z,z)}{e^{(n-1)r/2}} dr = c
\end{equation}
where $c = c(\G)$ is a constant depending only on $\Gamma$ (and $c=0$ for $\Gamma$ cocompact). 

The almost periodicity of the normalized error term in the $r$-variable is crucial (see \cite{cherubini}), but this property fails for the mean value in the $X$-variable.
Although the error term attains large growth when $n \geq 4$, for $n=2$ and $3$ the square root cancellation (\ref{conjecturen}) is far open and difficult. As in the Euclidean case (see \cite{kelmer}), one may expect that (\ref{conjecturen}) holds for a {\lq generic\rq} lattice $\Gamma \subset \so$ and for {\lq most\rq} points $z$ and $w$. 

Let $\G$ be cocompact and $e (T,z)$ denote the following modified mean average of the error term in the $X$-variable: 
\begin{equation} \label{meaninxdimension2}
e(T,z) = e_{\G}(T,z) :=  \frac{1}{T} \int_{\sqrt{T}}^{T}  \frac{E_{\Gamma}(x;z,z)}{x^{(n-1)/2}} dx.
\end{equation}
This normalized error term is no longer almost periodic in the $X$-variable (see \cite{chatz}), hence we expect a less normal behaviour for $e(T,z)$. In $2$ dimensions we have $e(T,z) = O(1)$, but Chatzakos \cite{chatz} proved that the limit 
$ \lim_{T \to \infty}e(T,z)$
may not exist (see also Remark \ref{remark interval}). 

We study the behaviour of the modified mean value $e(T,z)$ for $n \geq 3$.  
We prove that the normalized error term $e(T,z)$ does not have finite mean value for large dimensions, which can be compared with the mean value result (\ref{mnrt}) of Phillips and Rudnick in the $r$-variable. 

\begin{theorem}\label{result1}
Let $\Gamma  \subset  \so$ be a cocompact lattice acting on $\mathbb{H}^n$.  Then:
\\(a) for $n=3$ we have
\begin{equation*} e_{\G} (T,z)= \Omega \left(\log \log \log T \right),
\end{equation*}
\\b) for $n \geq 4$ we have
\begin{equation*} 
 e_{\G} (T,z) = \Omega \left( (\log \log T)^{\frac{1}{2}-\frac{3}{2n}-\epsilon} \right)
\end{equation*}
for all $\epsilon>0$.
\end{theorem}

\begin{remark}\label{remark interval}
In \cite{chatz} Chatzakos studied the behaviour of the average of the mean error term in the full interval $[2,T]$ instead of $[\sqrt{T}, T]$. For $n=2$ there are no technical differences between the two average terms. However, for $n \geq 3$, we choose to average the error in $[\sqrt{T}, T]$ since this simplifies our calculations. In particular, studying the full interval $[2,T]$ yields a factor that does not get absorbed in the error term, whereas studying the dyadic interval $[T,2T]$ yields oscillatory factors, with alternating sign for arbitrarily large $t$. Our average does not affect our result in the sense that we work in an interval of asymptotic length $T$.
\end{remark}

\subsection{\texorpdfstring{$\Omega$}{TEXT}-results for the second moment} 
In the classical Gauss circle problem Cram\'er \cite[(eq.)~5]{cramer} proved an asymptotic formula for the second moment of the error term with a power saving.
Cram\'er's result indicates a very regular behaviour of the Euclidean error term, as the extra $X^{\epsilon}$-factors disappear in the main asymptotic of the second moment. An average result of this kind is called a \textit{Cram\'er-type result}.
For analogous results in higher Euclidean dimensions 
see \cite{nowak, jarnik}.

For $\GmodH^2$, Phillips and Rudnick studied the second moment in the $r$-variable. They proved that 
\begin{eqnarray} \label{secondmomentphillipsrudnick}
\frac{1}{R} \int_{0}^{R} \left|\frac{E_{\Gamma}(e^r;z,z)}{e^{r/2}}\right|^2 dr = \Omega(1),
\end{eqnarray}
and speculated that a Cram\'er-type result in the $r$-variable may hold. For the second moment in the $X$-variable, the $\Omega$-results of \cite{chatz} and the large sieve of Chamizo \cite{cham1, cham2} imply that
\begin{eqnarray} \label{chamizobound}
\frac{1}{T} \int_{2}^{T} \left|\frac{E_{\Gamma}(x;z,z)}{x^{1/2}}\right|^2 dx 
\end{eqnarray}
is $\Omega \left(1\right)$ and $O( \log^2 T)$. The upper bound was subsequently improved by Cherubini \cite{cherubini} to $O(\log T)$ using the almost periodicity of the error term. If (\ref{chamizobound}) is $O(1)$, then  (\ref{secondmomentphillipsrudnick}) is also $O(1)$. However, although we expect a Cram\'er-type result in the $r$-variable to be true, the boundedness of (\ref{chamizobound}) is doubtful. 

For $\GmodH^n$, with $n \geq 3$, the large sieve inequalities give worse upper bounds for the second moment than in $2$ dimensions, due to the abundance of eigenvalues (see Laaksonen \cite{laaksonen} for the case $n=3$). 
As a straightforward corollary of Theorem \ref{result1}, we prove the following $\Omega$-results.

\begin{corollary}\label{result2}
Let $\Gamma$ be a cocompact lattice acting on $\mathbb{H}^n$.  Then:
\\a) for $n=3$ we have
\begin{eqnarray*}
\frac{1}{T} \int_{2}^{T} \left|\frac{E_{\Gamma}(x;z,z)}{x}\right|^2 dx = \Omega \left( (\log \log \log T)^2 \right),
\end{eqnarray*}
\\
b) for $n \geq 4$ we have
\begin{eqnarray*}
\frac{1}{T} \int_{2}^{T} \left|\frac{E_{\Gamma}(x;z,z)}{x^{(n-1)/2}}\right|^2 dx = \Omega \left((\log \log T)^{1-\frac{3}{n}-\epsilon}\right)
\end{eqnarray*}
for all $\epsilon>0$.
\end{corollary}

\begin{remark}
The most interesting case of Corollary \ref{result2} is part (a), where we expect the upper bound $O(X^{1+\epsilon})$ to be true. Hill and Parnovski \cite{hillparn} studied a (spatial) variance of the error term $E_\G(X;z,w)$, where they fix the one point and take the average over the other. They deduced an explicit formula for the second moment where the error term has a saving of the extra $X^{\epsilon}$-factors. Their results \cite[Th.~1, Cor.~1]{hillparn} can be compared with the work of Laaksonen \cite{laaksonen} and our Corollary \ref{result2}.
\end{remark}

\subsection*{Acknowledgements}
I would like to thank my supervisor Dimitrios Chatzakos for helpful discussions. During this work I was supported by the Hellenic Foundation for Research and Innovation (H.F.R.I.) under the “3rd Call for H.F.R.I. Research Projects to support Faculty Members \& Researchers” (Project Number: 25622).

\section{Background} 

\subsection{The contribution of the main term \texorpdfstring{$M_{\Gamma}(X; z,w)$}{TEXT}.}\label{contribution of main term}
For $\Gamma$ a cofinite group acting on $\H^n$ we can decompose the main term $M_{\Gamma}(X; z,w)$ in the asymptotic formula of Theorem \ref{mainformula} as
\begin{eqnarray} \label{definitionMG}
M_\Gamma(X;z,w)=\sum_{\frac{n-1}{2} \leq s_j \leq n-1} \pi^{\frac{n-1}{2}} \frac{\Gamma \left(s_j-\frac{n-1}{2}\right)}{\Gamma(s_j+1)} \phi_j(z) \overline{\phi_j(w)} X^{s_j} + A_{\Gamma}(X; z,w),
\end{eqnarray}
where following \cite{phirud}, the term $A_\Gamma(X;z,w)$ is a well understood quantity. More precisely, for $n \geq 3$ it can be split as a finite sum
\begin{eqnarray*}
A_{\Gamma}(X; z,w)= \sum_{\frac{n-1}{2} \leq s_j \leq n-1} a(X; s_j ,z,w), 
\end{eqnarray*}
where for every $s_j > (n-1)/2$ the term $a(X;s_j,z,w)$ contains the secondary contributions of the eigenvalue $\lambda_j < (n-1)^2/4$ (if this eigenvalue exists) and satisfies 
\begin{eqnarray*}
a(X;s_j,z,w) \ll X^{s_j-2}+o(X^{(n-1)/2})
\end{eqnarray*}
\cite[eq.~(4.3)]{phirud}. The term $a(X;(n-1)/2,z,w)$ includes the total contribution of the eigenvalue $\lambda_j = (n-1)^2/4$ and satisfies 
\begin{eqnarray*} 
a \left(X;\frac{n-1}{2},z,w\right) = O(X^{(n-1)/2+\epsilon}).
\end{eqnarray*}
Overall, we conclude that the total contribution of the secondary error terms is estimated as
\begin{eqnarray}\label{secondary}
A_{\G} (X;z,w)=O(X^{n-3}+X^{\frac{n-1}{2}+\epsilon}). 
\end{eqnarray}
We refer to \cite[sect. 4]{phirud} or \cite[Prop.3, p.~479]{hillparn} for a detailed explanation of this term (the two references do not agree because of a typo in \cite[p.103 eq.~(4.3)]{phirud}).

\subsection{Weyl's law} 

We write $s_j = (n-1)/2 + it_j$, hence $\lambda_j = (n-1)^2/4 + t_j^2$. We will refer to the following two main results describing the asymptotic behaviour of the eigenvalues $\lambda_j$ and of the eigenfunctions $\phi_j(z)$ as $j \to \infty$. The first result is Weyl's law for $ \Gamma \backslash \mathbb{H}^n$.

\begin{theorem}[Weyl's law] \label{weylslaw} Let $\GmodHn$ be a $n$-dimensional compact manifold. Then as $T \to \infty$ we have 
\begin{eqnarray*}
\# \left\{ j: |t_j| \leq T \right\} \sim \frac{ \vol(\Gamma \backslash \mathbb{H}^n)}{(4 \pi)^{n/2} \Gamma \left(\frac{n}{2} +1\right)} T^n.
\end{eqnarray*}
\end{theorem}

The local version of Theorem \ref{weylslaw} is the so-called \textit{local Weyl's law} for the space $\L(\Gamma \backslash \mathbb{H}^n)$ (see \cite[p.~103, eq.~(4.7)]{phirud}). For our purposes it suffices to state the result in the cocompact case. 
\begin{theorem}[Local Weyl's law]  Let $\G$ be a cocompact group. Then, for every $z \in \mathbb{H}^n$, as $T \to \infty$ we have
\begin{eqnarray*}
\sum_{|t_j| \leq T} |\phi_j(z)|^2 \sim c T^n,
\end{eqnarray*}
where the constant $c=c(z)$ depends only on the point $z$ and the dimension. 
\end{theorem}

\subsection{The Selberg/Harish-Chandra transform}

We choose the kernel $k$ depending on the hyperbolic distance between two points to be the characteristic function
\begin{eqnarray}\label{kernel}
 k_X(u(z, w)) = k_X (u) = \left\{ \begin{array}{rcl}
1, & \mbox{for} & u\leq ( X-2)/4, 
\\ 0, & \mbox{for} & u>( X-2)/4,
\end{array} \right. 
\end{eqnarray}
where $u(z,w)$ is the standard point-pair invariant function satisfying $2  u(z,w) +1 =\cosh d(z,w)$
The $n$-dimensional Selberg/Harish-Chandra transform $h_X(t)$ of the kernel defined in (\ref{kernel}) is given by the formula
\begin{equation}\label{spectral function}
    h_X(t)=c_n\int_{-r}^{r}(\cosh r-\cosh u)^{\frac{n-1}{2}}e^{itu}\,du,
\end{equation}
where $c_n$ is a specific constant depending only on the dimension $n$ (see \cite[p.~102]{phirud}).
From now on, we will denote by $c_n$ any multiple of this constant (not necessarily the same) depending only on the dimension $n$.

As before, let us set $X=2 \cosh r$. Using \cite[p.~962, eq.~(8.715.1)]{gradry} we can write $h_{X}(t)$ in the form
\begin{eqnarray*}
h_{X}(t) = c_n(\sinh r)^{n/2} P_{-1/2+it}^{-n/2} (\cosh r), 
\end{eqnarray*}
where $P_{\nu}^{\mu}$ is the associated Legendre function.
By formula \cite[p.~971,eq.~(8.776.1)]{gradry} we calculate, for $X \gg 1$:
\begin{equation} \label{httransform}
h_X(t) = c_n   \Re \left(  \frac{\G(it)}{\G \left(\frac{n+1}{2} +it \right)} X^{i t} \right) \left( X^{\frac{n-1}{2}} + O \left( X^{\frac{n-5}{2}}\right) \right).
\end{equation} 
Stirling's formula \cite[p.~895, eq.~(8.328.1)]{gradry} reads $\lim_{|t| \to +\infty} |\Gamma (x+it)| \sim  \sqrt{2\pi} e^{-\pi|t|/2} |t|^{x-1/2}$, thus for $X, |t| \gg 1$ we have the asymptotic
 \begin{equation} \label{httransform2}
h_X(t) =  c_n \frac{  X^{\frac{n-1}{2}}}{|t|^{\frac{n+1}{2}}} \Re(\theta_n X^{it}) + O\left( \frac{X^{(\frac{n-5}{2})}}{|t|^{\frac{n+3}{2}}}\right),
\end{equation}
where $\theta_n$ is an angular constant depending only on the dimension $n$. Additionaly using formula \cite[p.~963,eq.~(8.723.1)]{gradry} we can express $h_{X}(t)$ as a sum of conjugates in the form 
\begin{equation}\label{spectralsumofconj}
    h_{X}(t)=\tilde{A}(X,t)e^{itr}+\tilde{B}(X,t)e^{-itr},
\end{equation}
where in terms of the Gamma and Hypergeometric functions we have
\begin{equation}\label{coefficients for conj}
    \tilde{A}(X,t)=c_n(\sinh r)^{\frac{n-1}{2}}\frac{\Gamma(it)}{\Gamma(\frac{n+1}{2}+it)}F\left(\frac{1-n}{2},\frac{n+1}{2},1-it,\frac{1}{1-e^{2r}}\right)
\end{equation}
and $\tilde{B}(X,t)=\tilde{A}(X,-t).$
On the other hand in the spirit of \cite[Lemma~2.4]{cham2} for any $t\in\mathbb{C}$ we have the uniform upper bound:
\begin{equation}\label{spectraluniform}
    h_X(t)=O(X^{\frac{n-1}{2}+|\Im t|} \log X). 
\end{equation}

\section{Local average for the lattice counting}\label{localaverage}

\subsection{The automorphic kernel and heuristics for the lattice point counting}\label{kernelsubsection}

Let us recall heuristically the spectral approach to the hyperbolic lattice point counting. Assume $\Gamma$ is a cocompact lattice in $\so$. We need to apply the pre-trace formula (see for instance \cite{parnovski}) for $\GmodHn$ to the automorphic kernel
\begin{eqnarray*} 
K(z,w) = \sum_{\gamma \in \Gamma} k (z,\gamma w),
\end{eqnarray*}
where $k$ is a suitable kernel depending on the hyperbolic distance $d(z, \gamma w)$. Then, the pre-trace formula allows us to express $K(z,w)$ as a spectral expansion
\begin{eqnarray*} 
K(z,w) = \sum_{\lambda_j} h (t_j)  \phi_j(z) \overline{\phi_j(w)}
\end{eqnarray*}
over the Laplacian eigenvalues $\lambda_j = (n-1)^2/4 + t_j^2$. Here, $h(t)$ is the $n-$dimensional Selberg/Harish-Chandra transform of the kernel $k$. If we could choose the kernel $k_X(u(z, w))$ defined in (\ref{kernel}), in that case we can easily verify that
\begin{eqnarray}\label{latticekernel}
N_\Gamma(X;z,w) =K_X(u(z,w))= \sum_{\gamma \in \Gamma} k_{X} (u(z,\gamma w)).
\end{eqnarray}
Applying the pre-trace formula directly to (\ref{latticekernel}), setting $z=w$ and subtracting the contribution of the small eigenvalues (see Theorem \ref{mainformula}), we would get a spectral expansion for the error term $E_{\Gamma}(X;z,z)$ over the real, non-zero Satake parameters $t_j$: 
\begin{eqnarray*} 
E_{\Gamma} (X;z,z)  \sim \sum_{|t_j| >0} h_X(t_j)  |\phi_j(z)|^2.
\end{eqnarray*}
However, the characteristic kernel is not an admissible kernel in order to apply the pre-trace formula, hence we need to consider smooth approximations. 

For the local average result, we will define appropriate smooth approximations $k_{\pm} (u)$ of $k(u)$ satisfying the inequality 
\begin{equation}\label{firstinequality}
    k_{-}(u)\leq k_{X}(u)\leq k_{+}(u).
\end{equation}
Then the corresponding automorphic kernels $N^{\pm}_{\G}$ defined as 
\begin{equation}\label{pmautokernel}
    N^{\pm}_{\G}(X;z,w):=\sum_{\gamma\in\G}k_{\pm}(u(z,\gamma w))
\end{equation}
also satisfy an inequality similar to (\ref{firstinequality}):
\begin{equation}\label{ineqkernels}
N^{-}_{\G} (X;z,w) \leq N_{\G} (X;z,w)
\leq   N^{+}_{\G} (X;z,w).
\end{equation}
Therefore in order to prove Theorem \ref{local average main theorem} it suffices to bound the local average
\begin{equation*}
    \int_{\GmodH^n} f(z) E_{\G}^{\pm}(X;z,z) \, d \mu(z),
\end{equation*}
where 
\begin{equation}\label{errorpm}
    E_{\G}^{\pm}(X;z,w)=N_{\G}^{\pm}(X;z,w)-M_{\G}^{\pm}(X;z,w)
\end{equation}
and $M_{\G}^{\pm} (X;z,w)$ are sufficiently good approximations of the main term $M_{\G} (X;z,w)$ defined by (\ref{definitionMG}). If we set $z=w$ in Selberg's pre-trace formula for the automorphic kernel defined in (\ref{pmautokernel}), subtract $M_{\G}^{\pm}(X;z,z)$ and integrate against $f$ we obtain
\begin{equation}\label{mainerrorestimate}
\int_{\Gamma\setminus\mathbb{H}^n}f(z)
E_{\Gamma}^{\pm} (X; z, z) \, d\mu(z) = \Bar{f} \sum_{|t_j|>0}h_{\pm}(t_j)+\sum_{|t_j|>0}h_{\pm}(t_j)\left(\int_{\Gamma\setminus\mathbb{H}^n}f(z)\,d\mu_ j(z)-\Bar{f}\right),
\end{equation}
where $f$ is a smooth compactly supported function on $\mathcal{M}=\GmodHn$ and $\Bar{f}$ as defined in (\ref{overlinef}).
Here $h_{\pm}$ is the Selberg/Harish-Chandra transform of kernel $k_{\pm}$, which we will describe in detail in subsection \ref{smoothsubsection}. 
Moreover $\,d\mu_j(z)$ are the measures defined in (\ref{measures}) for the Riemannian manifold $\mathcal{M}=\Gamma\backslash\H^n$. Inequality (\ref{ineqkernels}) together with the definition of $E_\G^\pm$ will lead to the proof of Theorem \ref{local average main theorem}.

\subsection{Smooth approximations}\label{smoothsubsection}

For smoothing purposes we will convolute the kernel $k_X$,  with the family of test functions 
\begin{equation}\label{mollifier}
    k_{\delta}(u)=\frac{1}{ \hbox{vol} (B_n(\delta))} \cdot \chi_{[0,\frac{\cosh \delta -1}{2}]}(u),
\end{equation}
for $0<\delta<1$ a (sufficiently small) parameter, which will be later chosen as a negative power of  $X$, and
\begin{equation}\label{n-ball radius}
\hbox{vol}\left(B_n(r)\right)=\frac{2\pi^{n/2}}{\Gamma(n/2)}\int_{0}^{r}(\sinh t)^{n-1}\,dt
\end{equation}
is the volume of the $n$-dimensional hyperbolic ball of radius $r$. Even though expression (\ref{n-ball radius}) can be computed in terms of special functions (see \cite[Corollary~2]{hillparn}) we will only need the asymptotics for small values of $r$, that is
\begin{equation*}
    \hbox{vol}\left(B_n(r)\right)\sim \frac{2\pi^{n/2}}{n\G(n/2)}r^{n}.
\end{equation*}

\begin{lemma}\label{hdeltalemma}
    Let $k_\delta(u)$ be the kernel defined in \eqref{mollifier} and $h_\delta(t)$ its Selberg/Harish-Chandra transform given by \eqref{spectral function}. Then for $t\geq 1$ and $\delta\to 0$, we have that
    \begin{equation*}
    h_{\delta}(t)=O\left(\min (1,(\delta t)^{-(\frac{n+1}{2})})\right)
\end{equation*}
and also
    \begin{equation*}
     h'_{\delta}(t)=O\left(\frac{1}{t}\min (1,(\delta t)^{-(\frac{n-1}{2})})\right).
\end{equation*}
\end{lemma}
\begin{proof}
    Using the identity $\cosh \delta=2\sinh^2(\delta/2)+1$ and with the change of variable $u=\delta y$ we can rewrite $h_\delta(t)$ as
    \begin{equation*}
        h_\delta(t)=c_n2^{\frac{n-1}{2}}\delta^{1-n}\left(\sinh\frac{\delta}{2}\right)^{n-1}\int_{-1}^{1}\left(1-\left(\frac{\sinh(\delta y/2)}{\sinh(\delta/2)}\right)^2\right)^{\frac{n-1}{2}}e^{it\delta y}\,dy.
    \end{equation*}
    Applying dominated convergence  we have that if $\delta\to 0$ and $\delta t\to 0$ then 
    \begin{equation*}
       h_\delta(t)\sim c_n 2^{\frac{n-1}{2}}\delta^{1-n}\left(\sinh\frac{\delta}{2}\right)^{n-1}\int_{-1}^{1}(1-y^2)^{\frac{n-1}{2}}\,dy ,
    \end{equation*}
    where the integral can be explicitly computed in terms of the Gamma function but for our purposes it suffices to note that it is a constant depending only on $n$. Also if $\delta\to 0$ and $\delta t\in(0,+\infty)$ we have that 
    \begin{equation*}
        h_\delta(t)\sim c_n 2^{\frac{n-1}{2}}\delta^{1-n}\left(\sinh\frac{\delta}{2}\right)^{n-1}\int_{-1}^{1}(1-y^2)^{\frac{n-1}{2}}e^{it\delta y}\,dy,
    \end{equation*}
    where by \cite[eq. 8.411.10]{gradry} we can evaluate the integral in terms of Bessel functions and deduce
    \begin{equation*}
        h_\delta(t)\sim c_n 2^{\frac{n-1}{2}}\delta^{1-n}\left(\sinh\frac{\delta}{2}\right)^{n-1}\sqrt{\pi}2^{n/2}\G\left(\frac{n+1}{2}\right)\frac{J_{n/2}(\delta t)}{(\delta t)^{n/2}}.
    \end{equation*}
   Applying the bound $J_n(z)=O(\min(z^n,z^{-1/2}))$ for the Bessel function (see \cite[Appendix B4]{iwaniec}), the first part of the lemma follows. As for the second part, concerning the derivative of $h_\delta(t)$ in \eqref{spectral function}, differentiating with respect to $t$ and applying similar arguments as above, the result follows.
\end{proof}

Setting $\cosh Y=X/2$,
it is straightforward that $u\leq (X-2)/4\iff d\leq Y$. Let $h_{Y\pm\delta}$ be the Selberg/Harish-Chandra transform for the characteristic function
\begin{equation}\label{ypmdelta}
\chi_{\left[0,\frac{\cosh(Y\pm\delta)-1}{2}\right]}
\end{equation}
and define the convolution

\begin{equation*}
k_{\pm}:= \chi_{\left[0,\frac{(\cosh(Y\pm\delta)-1)}{2}\right]}\ast k_\delta.
\end{equation*}
It follows that inequality (\ref{firstinequality}) holds for the kernels $k_-,k_X,k_+$ and as for the Selberg/Harish-Chandra transform of $k_\pm$ we have by the convolution identity (see \cite[4. p.323]{cham2}, we omit the proof in $\H^n$) that 
\begin{equation}\label{hconvolution}
h_\pm=h_{Y\pm\delta}\cdot h_\delta.
\end{equation}

\subsection{Bounding the contribution of the small eigenvalues}

In this section we bound the difference of the main term $M_\G(X;z,z)$ and its approximation $M_{\G}^{\pm}(X;z,z)$ (see subsection \ref{kernelsubsection}). We begin by subtracting $M_{\G}^{\pm}(X;z,w)$ from inequality (\ref{ineqkernels}) and using (\ref{errorpm}) we deduce that for $z=w$, we have

\begin{equation*}
    E_{\G}(X;z,z) \ll E_{\G}^{\pm}(X;z,z)+O\left(M_{\G}^{\pm}(X;z,z)-M_{\G}(X;z,z)\right).
\end{equation*}
The term $M_{\G}^{\pm}(X;z,z)$ is the main term of the spectral expansion of $N_{\G}^{\pm}(X;z,z)$ and is given by
\begin{equation*}
M_{\G}^{\pm}(X;z,z)=\sum_{\frac{n-1}{2}\leq s_j\leq n-1 }h_\pm(it_j)|\phi_j(z)|^2,
\end{equation*}
with $s_j=\frac{n-1}{2}+it_j$. By the definition of $h_{\pm}$ we need to estimate $h_{Y\pm\delta}$ and $h_\delta$ for imaginary and bounded argument. Using \cite[eq.~4.3]{phirud} and
the fact that 
\begin{equation*}
    e^{Y\pm\delta}= X+O (X\delta),
\end{equation*}
we have that 
\begin{equation}\label{hypmimaginary}
h_{Y\pm\delta}(it)=c_nX^{|t|+\frac{n-1}{2}}+O(X^{|t|+\frac{n-1}{2}}\delta+X^{|t|+\frac{n-5}{2}}+X^{\frac{n-1}{2}}),
\end{equation}
for $|t_j|>0$ and $0<\delta<1$, with  $\delta>X^{-1}$. As for $t_j=0$ using equation (\ref{spectraluniform}), we have that $h_{Y\pm\delta}(0)=O(X^{\frac{n-1}{2}}\log X)$.
To estimate $h_\delta(it)$ we recall \cite[Prop.~2, Coro.~2]{hillparn} and we have
\begin{eqnarray*}
h_\delta(it)&=& \frac{1}{ \hbox{vol} (B_n(\delta))}\frac{\pi^{\frac{n-1}{2}}\G(\frac{n+1}{2})}{\G(n+1)}(2\sinh \delta)^{n}e^{(it-\frac{n+1}{2})\delta} \\
&&\times \, \, {}_{2}F_{1}\left(\frac{n+1}{2}-it,\frac{n+1}{2},n+1,1-e^{-2\delta}\right)\\
    &=&e^{(it-\frac{n-1}{2})\delta} \, \times \, \frac{{}_{2}F_{1}(\frac{n+1}{2}-it,\frac{n+1}{2},n+1,1-e^{-2\delta})}{{}_{2}F_{1}(1,\frac{n+1}{2},n+1,1-e^{-2\delta})},
\end{eqnarray*}
for all $\delta>0$ and all $t\in\mathbb{C}.$
Using the power series definition for the hypergeometric function (see \cite[eq.~9.100]{gradry}) it is straightforward that 

\begin{equation}\label{hdeltaimaginary}
    h_\delta(it)=1+O(\delta),
\end{equation}
for $0<\delta<1$ and $0\leq it\leq(n-1)/2.$
For $M_{\G}(X;z,z)$ we recall definition (\ref{definitionMG}). Note that the constant in expression (\ref{definitionMG}) agrees with the constant $c_n$ in equation (\ref{hypmimaginary}) and therefore the leading terms of $M_\G$ and $M_{\G}^{\pm}$ cancel each other out.
Using (\ref{definitionMG}), (\ref{secondary}) (\ref{hypmimaginary}), (\ref{hdeltaimaginary}) and 
summing over all $0\leq it_j\leq(n-1)/2$ we obtain a bound of the form

\begin{equation*}
    M_{\G}(X;z,z)-M_{\G}^{\pm}(X;z,z)=O\left(X^{n-1}\delta+X^{n-3}+X^{\frac{n-1}{2}+\epsilon}\right).
\end{equation*}
Therefore  for the error term we have the asymptotic
\begin{equation}\label{errorsbound}
E_{\G}(X;z,z) \ll E_{\G}^{\pm}(X;z,z)+O(X^{n-1}\delta+X^{n-3}+X^{\frac{n-1}{2}+\epsilon}).
\end{equation}
Thus, in order to estimate the local average of the error term we are left with the task to estimate the local average of $E_{\G}^{\pm}(X;z,z)$.

\subsection{Proof of Theorem \ref{local average main theorem}}
We begin this section by proving the following two Lemmas.
\begin{lemma}\label{hpminfo}
    Let $h_\pm(t)$ be as in \eqref{hconvolution} and $0<\delta<1$. We can express $h_\pm(t)$ in the form 
    \begin{equation*}
  h_\pm(t)=A (X,t,\delta,n)e^{it(Y\pm\delta)}+B(X,t,\delta,n)e^{-it(Y\pm\delta)},
\end{equation*}
where for $t>1$ the following bounds hold 

\begin{equation*}
\begin{gathered}
    A(X,t,\delta,n), \, \,B(X,t,\delta,n)\ll \frac{X^{\frac{n-1}{2}}}{|t|^{\frac{n+1}{2}}}\min \left(1,(\delta t)^{-(\frac{n+1}{2})}\right),\\
    \partial_t A(X,t,\delta,n), \, \,   \partial_t B(X,t,\delta,n) \ll \frac{X^{\frac{n-1}{2}}}{|t|^{\frac{n+3}{2}}}\min\left(1,(\delta t)^{-(\frac{n-1}{2})}\right).
    \end{gathered}
\end{equation*}

\end{lemma}
\begin{proof}
    We apply (\ref{spectralsumofconj}) and (\ref{coefficients for conj}) for $h_{Y\pm\delta}(t)$ defined as the Selberg/Harish-Chandra transform of the characteristic function defined in (\ref{ypmdelta}). Then we have an expression for $h_{Y\pm\delta}$ as a sum of conjugates, with an explicit form on the coefficients similar to (\ref{coefficients for conj}) and $Y\pm\delta$ in place of $r$. Having in mind that $e^{Y\pm\delta}\asymp X$, using Stirling's formula for the part involving the Gamma factors and the  power series definition of the Hypergeometric function (\cite[eq.~9.100]{gradry}) we derive bounds for the coefficients of the sum of conjugates expression for $h_{Y\pm\delta}$ and also for their derivatives with respect to $t$. From the above discussion and Lemma \ref{hdeltalemma} the result  follows.
\end{proof}

\begin{lemma}\label{quantum variance}
Let $n \geq 3$ and assume Conjecture \ref{weakerconjecture} with $q_n=n-2+\epsilon$. Then the following estimate holds:

\begin{equation*}
    \sum_{|t_j|>0}h_\pm(t_j)\left(\int_{\Gamma/\mathbb{H}^n}f(z)\,d\mu_j(z)-\Bar{f}\right)=O_{f}\left(X^{\frac{n-1}{2}}\delta^{\frac{3-n}{2}-\epsilon}\right).
\end{equation*}
\end{lemma}

\begin{proof} 
For $|t_j|< 1$ using (\ref{spectraluniform}), (\ref{hconvolution}) and Lemma \ref{hdeltalemma}, the contribution is $O(X^{(n-1)/2}\log X)$ and as for $|t_j|\geq 1$ by Cauchy-Schwarz inequality and assuming Conjecture \ref{weakerconjecture}, we have over dyadic intervals that

\begin{eqnarray}\label{2}
    &&\left|\sum_{T\leq|t_j|< 2T}h_\pm(t_j)\left(\int_{\Gamma/\mathbb{H}^n}f(z)\,d\mu_j(z)-\Bar{f}\right)\right| \nonumber \\
&&\leq
\left(\sum_{T\leq|t_j|< 2T}|h_\pm(t_j)|^2\right)^{1/2}\left(\sum_{T\leq|t_j|< 2T}\left|\int_{\Gamma/\mathbb{H}^n}f(z)\,d\mu_j(z)-\Bar{f}\right|^2\right)^{1/2}
\\ \nonumber \\
&&=O_{f}\left( T^{n-1+\epsilon}  \max_{T<|t|\leq 2T}|h_\pm(t)|\right). \nonumber
\end{eqnarray}
Using Lemma \ref{hpminfo}, we have that (\ref{2}) is
$$O_{f}\left(X^{(n-1)/2} T^{\frac{n-3}{2}+\epsilon}\min(1,(\delta T)^{-(n+1)/2}\right).$$
Splitting  cases for $T<\delta^{-1}$ and $T\geq\delta^{-1}$, with $T=2^k$ and summing for all $k\geq 0$ we have that 
\begin{equation*}
    \sum_{|t_j|\geq 1}h_\pm (t_j)\left(\int_{\Gamma\setminus\mathbb{H}^n}f(z)\,d\mu_j(z)-\Bar{f}\right)=O_{f }\left(X^{\frac{n-1}{2}}\delta^{\frac{3-n}{2}-\epsilon}\right).
\end{equation*}
As we mentioned earlier, the parameter $0<\delta<1$ will be chosen as a negative power of $X\to\infty$, hence without loss of generality we may assume that the above error term dominates over the error $O(X^{\frac{n-1}{2}}\log X)$ coming from $0<|t_j|<1$.
\end{proof}

We now want to estimate the term
$
\sum_{|t_j|>0}h_\pm(t_j)
$ in (\ref{mainerrorestimate}).
We have seen that for the contribution arising from  $|t_j|<1$ every term is $O(X^{\frac{n-1}{2}}\log X)$ and since there are finitely many such terms, it suffices to estimate the sum
$$\sum_{|t_j|\geq 1}h_\pm (t_j).$$
Summing over dyadic intervals and using Lemma \ref{hpminfo} we get 

\begin{equation*}
    \sum_{T\leq|t_j|< 2T}h_\pm (t_j)\ll\sum_{T\leq|t_j|< 2T}A (X,t_j,\delta,n)e^{(Y\pm\delta)it_j},
\end{equation*}
where by partial summation and definition (\ref{expsumsdef}) this is equal to 

\begin{eqnarray}\label{partialsummationterm}
A(X,2T,\delta,n)S(2T,e^{Y\pm\delta})-A(X,T,\delta,n)S(T,e^{Y\pm\delta})-
\\
\int_{T}^{2T}A'(X,u,\delta,n)S(u,e^{Y\pm\delta})\,du.\nonumber
\end{eqnarray}
Recalling Conjecture \ref{exp sums} and using Lemma \ref{hpminfo} we can estimate the above expression. The first two terms of (\ref{partialsummationterm}) are bounded by
\begin{eqnarray*}
O\left(X^{\frac{n-1}{2}+\epsilon}T^{\frac{n-3}{2}+\epsilon}\min\left(1,(\delta T)^{-(\frac{n+1}{2})}\right)\right),
\end{eqnarray*}
whereas the integral is bounded by 
\begin{eqnarray*}
O\left(X^{\frac{n-1}{2}+\epsilon}T^{\frac{n-3}{2}+\epsilon}\min\left(1,(\delta T)^{-(\frac{n-1}{2})}\right)\right).
\end{eqnarray*}
We conclude that 
\begin{equation*}
    \sum_{T\leq|t_j|< 2T}h_\pm (t_j)= O\left(X^{\frac{n-1}{2}+\epsilon}T^{\frac{n-3}{2}+\epsilon}\min\left(1,(\delta T)^{-(\frac{n-1}{2})}\right)\right).
\end{equation*}
Using again dyadic decomposition with $T=2^k,$ splitting for $T<\delta^{-1}$ and $T\geq\delta^{-1}$ and summing for $k\geq 0$ we deduce that
\begin{equation}\label{estimate pm }
    \sum_{|t_j|\geq 1}h_{\pm}(t_j)=O\left( X^{\frac{n-1}{2}+\epsilon}\delta^{\frac{3-n}{2}-\epsilon}\right).
\end{equation}
 By equation (\ref{mainerrorestimate}), the bound (\ref{estimate pm }) and  Lemma \ref{quantum variance}, we get the local average upper bound for the smooth error terms

\begin{equation*}
\int_{\Gamma\setminus\mathbb{H}^n}f(z)E_{\G}^{\pm}(X;z,z)\,d\mu(z)=O_{f}(X^{\frac{n-1}{2}+\epsilon}\delta^{\frac{3-n}{2}-\epsilon}).
\end{equation*}
Recalling (\ref{errorsbound}) for the relation between the error term and the smooth error terms, we obtain
\begin{equation*}
    \int_{\Gamma\setminus\mathbb{H}^n}f(z)E_{\G}(X;z,z)\,d\mu(z)=O_{f}(X^{n-1}\delta+X^{n-3}+X^{\frac{n-1}{2}+\epsilon}+X^{\frac{n-1}{2}+\epsilon}\delta^{\frac{3-n}{2}-\epsilon}).
\end{equation*}
Choosing $\delta=X^{-1+\epsilon}$ completes the proof of Theorem \ref{local average main theorem}.

\section{Proof of the \texorpdfstring{$\Omega$-result}{TEXT}} 

Before giving the proof of Theorem \ref{result1}, we can make some comments for the method of the proof. As in section \ref{localaverage} we cannot apply the pre-trace formula directly to the kernel $k_X$ as it is not of rapid decay. For $n=2$, convergence problems disappear when we apply the pre-trace formula to the smoothed error $e(T,z)$ given by (\ref{meaninxdimension2}), however for $n\geq 3$ we obtain 
\begin{eqnarray} \label{meanvalue4}
e(T,z) = \sum_{t_j >0}  |\phi_j(z)|^2 \frac{1}{T} \int_{\sqrt{T}}^{T} \frac{h_x(t_j)}{x^{(n-1)/2}} dx.
\end{eqnarray}
Applying local Weyl's law for $\L(\Gamma \backslash \mathbb{H}^n)$ and using estimate (\ref{httransform2}) we deduce that the principal series in (\ref{meanvalue4}) does not converge. For this reason, we need to smooth out the error term further. We do so by mollifying $e(T,z)$ with an appropriately chosen family of test functions. We can now proceed to the proof of our first $\Omega$-result.
\begin{proof} (of Theorem \ref{result1}) Let $\psi \geq 0$ be a smooth, even and compactly supported function in $[-1,1]$,  such that 
\begin{eqnarray*}
\int_{-\infty}^{+\infty} \psi(x) \,e^{-itx} dx  := \hat{\psi} (t) \geq 0
\end{eqnarray*}
and $\int_{-\infty}^{+\infty} \psi(x) dx = 1$.  We also define the family of functions $\psi_{\epsilon}(x) = \epsilon^{-1} \psi(x/\epsilon)$, for every $0<\epsilon<1$. We have $0 \leq \hat{\psi}_{\epsilon}(t) \leq 1$ and $\hat{\psi}_{\epsilon}(0) = 1$. 
Using integration by parts it is straightforward that 
\begin{equation}\label{psi estimate}
    \hat{\psi}_{\epsilon}(t) = O_k \left(\frac{1}{\epsilon^k(1+|t|)^k}\right)
\end{equation}
for any $k\in\mathbb{Z}$ and  $t\in\R$.
We mollify $e(T,z)$ with the family $\{\psi_{\epsilon}\}$, i.e. we consider the convolution
\begin{eqnarray*} 
e_{\G,\epsilon}^* (R,z)
:= \int_{-\infty}^{+\infty} \psi_{\epsilon} (R-Y) \, e (e^Y,z) dY.
\end{eqnarray*}

Our $\Omega$-result for $e(T,z)$ will follow from a $\Omega$-result for $e_{\G,\epsilon}^* (R,z)$. Applying the pre-trace formula and using expression (\ref{httransform}) for the Selberg/Harish-Chandra transform we get the expansion

\begin{eqnarray*}
e_{\G,\epsilon}^* (R,z)&=&
\sum_{|t_j| >0}  c_n    |\phi_j(z)|^2 
 \Re \left(    \frac{\G(it_j)}{\G \left(\frac{n+1}{2} + it_j \right)}   \int_{-\infty}^{+\infty}    \frac{ \psi_{\epsilon} (R-Y) }{e^Y}  \int_{e^{Y/2}}^{e^Y} x^{i t_j}  dx  dY \right)\\
 &&+O \left(\sum_{|t_j| >0}    |\phi_j(z)|^2 
 \Re \left(    \frac{\G(it_j)}{\G \left(\frac{n+1}{2} + it_j \right)}   \int_{-\infty}^{+\infty}    \frac{ \psi_{\epsilon} (R-Y) }{e^Y}  \int_{e^{Y/2}}^{e^Y} x^{i t_j-2}  dx  dY \right)\right).
\end{eqnarray*}
The compact support of $\psi(x)$ implies that we integrate for $R\asymp Y$. Further, using the  properties of $\psi(x)$, applying Stirling's asymptotic formula and (\ref{psi estimate})  we compute
\begin{equation} \label{finalasymptoticforaverage}
\begin{aligned}
e_{\G,\epsilon}^* (R,z)=& \sum_{|t_j| >0}  c_n   |\phi_j(z)|^2 \hat{\psi}_{\epsilon}(t_j)  \Re \left( \frac{\G(it_j)}{\G \left(\frac{n+1}{2} + it_j \right) (1+it_j)}   e^{i t_j R} \right)\\
&+ O_{k} \left( e^{-R/2}  \sum_{|t_j| >0}  \frac{ \epsilon^{-k} }{(1+|t_j|)^{\frac{n+3}{2}+k}}   |\phi_j(z)|^2   \right). 
\end{aligned}
\end{equation}
One can easily verify that for every $\epsilon>0$ the principal series in (\ref{finalasymptoticforaverage}) converges  for $k > (n-3)/2$. For such a $k$ and
for $A \gg 1$, we truncate the principal series in (\ref{finalasymptoticforaverage}); the tail of the series in the main term for $t_j>A$ is bounded by
\begin{eqnarray*}
\sum_{t_j > A} \frac{ \hat{\psi}_{\epsilon}(t_j) }{(1+|t_j|)^{\frac{n+3}{2}}}   |\phi_j(z)|^2  =  O_{k} \left(\epsilon^{-k} A^{\frac{n-3}{2} -k} \right)
\end{eqnarray*}
Similarly, the error term of (\ref{finalasymptoticforaverage}) is bounded by $O_k \left(\epsilon^{-k} e^{-R/2} \right)$. The initial part of the series for $t_j \leq A$ can be handled using the Dirichlet's box principle.
\begin{lemma} [Dirichlet's box principle, \cite{phirud}] \label{dirichletsboxprinciple} Let $r_1, r_2, ... , r_m$ be $m$ distinct real numbers and $M>0$, $N>1$. Then, there is an $R$, $M \leq R \leq M N^m$, such that
\begin{eqnarray*}
|e^{ir_jR} -1| < \frac{1}{N}
\end{eqnarray*}
for all $j=1,...,m$.
\end{lemma}
Applying Lemma \ref{dirichletsboxprinciple} (Dirichlet's box principle) in the initial part of (\ref{finalasymptoticforaverage}) for a sufficiently large $N$ (we will choose it later) and for $t_j \leq A$ (hence $m \asymp A^n$) and using local Weyl's law we find a large $R \leq N^{A^n}$ such that 
\begin{eqnarray} \label{finalformoftheerror}
e_{\G,\epsilon}^* (R,z)&=& \sum_{t_j \leq A }  c_n    |\phi_j(z)|^2 \hat{\psi}_{\epsilon}(t_j)  \Re \left( \frac{\G(it_j)}{\G \left(\frac{n+1}{2} + it_j \right) (1+it_j)} \right)\nonumber  \\
&&+ O_{k} \left( N^{-1}  A^{\frac{n-3}{2} +o(1)} + \epsilon^{-k} A^{\frac{n-3}{2} -k}   +    \epsilon^{-k} e^{-R/2} \right),
\end{eqnarray}
(where for the first term in the error we simply use $|\hat{\psi}_{\epsilon}(t_j)|\leq 1$).
Using the balance $N= A^{n/2}$ and $ \epsilon^{-k} = A^{ k - \frac{n-3}{2}}$ (then we get $ \epsilon^{-k} \ll e^{R/2}$), the error term in (\ref{finalformoftheerror}) is bounded by $O (1)$. The bound $R \leq N^{A^n}$ implies

\begin{eqnarray} \label{balanceerror}
\log R \ll n \epsilon^{-\frac{2kn}{2k-(n-3)}} \log ( \epsilon^{-1}).
\end{eqnarray}
Further, there exists a $\tau \in (0,1)$ such that $\hat{\psi}(x) \geq 1/2$ whenever $|x| \leq \tau$. Using Stirling's asymptotics and the relation $\hat{\psi}_{\epsilon}(t_j) = \hat{\psi} (\epsilon t_j)$, from (\ref{finalformoftheerror}) we deduce there are arbitrarily large values of $R$ satisfying the bound
\begin{eqnarray*}
e_{\G,\epsilon}^* (R,z)\gg \sum_{t_j \leq \tau/\epsilon } \frac{ |\phi_j(z)|^2}{(1+|t_j|)^{\frac{n+3}{2}}}. 
\end{eqnarray*}
Using local Weyl's law and partial summation, we get for $n=3$ the lower bound
\begin{eqnarray*}
 \sum_{t_j \leq \tau/\epsilon } \frac{ |\phi_j(z)|^2}{(1+|t_j|)^{\frac{n+3}{2}}} \gg \log(\epsilon^{-1}), 
\end{eqnarray*}
and the balance (\ref{balanceerror}) gives $e_{\G,\epsilon}^* (R,z)= \Omega(\log \log R)$, hence $e (T,z) = \Omega(\log \log \log T)$. Similarly for $n \geq 4$ we deduce
\begin{eqnarray*}
 \sum_{t_j \leq \tau/\epsilon} \frac{ |\phi_j(z)|^2}{(1+|t_j|)^{\frac{n+3}{2}}} \gg  \epsilon^{-\frac{(n-3)}{2}}  
\end{eqnarray*}
and the balance (\ref{balanceerror}) gives $e_{\G,\epsilon}^* (R,z)= \Omega \left( (\log R)^{\frac{n-3}{2n} - \delta}\right)$ for every $\delta>0$, hence $e (T,z) = \Omega \left( (\log \log T)^{\frac{n-3}{2n} - \delta}  \right)$.
\end{proof}

We can now prove the $\Omega$-result for the second moment of the normalized error term.

\begin{proof} (of Corollary \ref{result2}) This is an immediate application of Cauchy-Schwarz inequality: for any $n \geq 2$ we get
\begin{eqnarray*}
\left( \int_{\sqrt{T}}^{T}  \frac{E_{\Gamma}(x;z,z)}{x^{(n-1)/2}} dx  \right)^2 &\leq& \left( \int_{\sqrt{T}}^{T} 1 dx \right)     \left( \int_{\sqrt{T}}^{T} \left|\frac{E_{\Gamma}(x;z,z)}{x^{(n-1)/2}}\right|^2 dx \right) \\
&\ll&  T  \left( \int_{\sqrt{T}}^{T} \left|\frac{E_{\Gamma}(x;z,z)}{x^{(n-1)/2}}\right|^2 dx \right),
\end{eqnarray*}
hence
\begin{eqnarray*}
\left( \frac{1}{T} \int_{\sqrt{T}}^{T}  \frac{E_{\Gamma}(x;z,z)}{x^{(n-1)/2}} dx  \right)^2 \ll  \frac{1}{T}  \int_{\sqrt{T}}^{T} \left|\frac{E_{\Gamma}(x;z,z)}{x^{(n-1)/2}}\right|^2 dx \leq  \frac{1}{T}  \int_{2}^{T} \left|\frac{E_{\Gamma}(x;z,z)}{x^{(n-1)/2}}\right|^2 dx.
\end{eqnarray*}
The statement follows from Proposition \ref{result1}.
\end{proof}


\begin{thebibliography}{99}


\bibitem{balkanova}
O.~Balkanova, D.~Chatzakos, G.~Cherubini, D.~Frolenkov and N.~Laaksonen,
\emph{Prime geodesic theorem in the $3$-dimensional hyperbolic space.}
Trans. Amer. Math. Soc. \textbf{372} (2019), no.~8, 5355--5374. 

\bibitem{balkanova3}
O.~Balkanova and D.~Frolenkov,
\emph{Bounds for a spectral exponential sum.}
J. Lond. Math. Soc.~(2) {\bf 99} (2019), no.~2, 249--272.


\bibitem{balkanova4}
O.~Balkanova and D.~Frolenkov,
\emph{Sums of Kloosterman sums in the prime geodesic theorem.}
Q. J. Math. {\bf 70} (2019), no.~2, 649--674.




\bibitem{balkanova2} O.Balkanova, D.Frolenkov
\emph{Prime geodesic theorem for the Picard manifold.} Adv. Math. {\bf 375} (2020), 107377, 42 pp.



\bibitem{biro}A. Bir\'o,
\emph{Local average of the hyperbolic circle problem for Fuchsian groups.} Mathematika {\bf 64} (2018), no.~1, 159--183.

\bibitem{biro2} A. Bir\'o, 
\emph{On a generalization of the Selberg trace formula.} Acta Arithmetica, 87 (4) (1999), 319-338.

\bibitem{biro3} A. Bir\'o,
\emph{Local square mean in the hyperbolic circle problem.}
to appear in Algebra and Number Theory, \url{https://arxiv.org/abs/2403.16113}.

\bibitem{blomer}V.Blomer and C.Lutsko,
\emph{Hyperbolic lattice counting points in unbounded rank.} J. Reine Angew. Math. {\bf 812} (2024), 257--274.

\bibitem{cham1} F. Chamizo,
\emph{The large sieve in Riemann surfaces.} Acta Arithmetica, 77, no. 4, 303--313, 1996. 

\bibitem{cham2} F. Chamizo,
\emph{Some applications of large sieve in Riemann surfaces.} Acta Arithmetica, 77, no. 4, 315--337, 1996..

\bibitem{chatz} D. Chatzakos,
\emph{$\Omega$-results for the hyperbolic lattice point problem.} Proc. Amer. Math. Soc. {\bf 145} (2017), no.~4, 1421--1437.

\bibitem{chatz2} D. Chatzakos, G.Cherubini, N.Laaksonen,
\emph{Second moment of the prime geodesic theorem for $\pslzi$.} Math. Z. {\bf 300} (2022), no.~1, 791--806.

\bibitem{chatz3} D. Chatzakos, G. Cherubini, St. Lester and M. Risager,
\emph{The hyperbolic circle problem over Heegner points.} (2025), \url{https://arxiv.org/abs/2506.13883}

\bibitem{cherubini2} G. Cherubini, C.Katsivelos,
\emph{Local average in the Hyperbolic sphere problem.} (2025)
\url{https://arxiv.org/abs/2503.20455}.


\bibitem{cherubini} G. Cherubini,
\emph{Almost periodic functions and hyperbolic counting.} Int. J. Number Theory {\bf 14} (2018), no.~9, 2343--2368.

\bibitem{cramer} 
\emph{H. Cram\'er, \"Uber zwei S\"atze des Herrn G. H. Hardy.} Math. Z. {\bf 15} (1922), no.~1, 201--210.

\bibitem{duke}  W. Duke, Z. Rudnick and P. Sarnak,
\emph{Density of integer points on affine homogeneous varieties.} Duke Math. J. 71, no. 1, 143--179, 1993.



\bibitem{fricker} F. Fricker,
\emph{Ein Gitterpunktproblem im dreidimensionalen hyperbolischen Raum.} Comment. Math. Helv. 43, 402--416, 1968.



\bibitem{nowak} V.~C. Garc\'ia and W.~G. Nowak, 
\emph{A mean-square bound concerning the lattice discrepancy of a torus in $\mathbb{R}^3$.} Acta Math. Hungar. {\bf 128} (2010), no.~1-2, 106--115; 

\bibitem{good} A. Good,
\emph{Local analysis of Selberg's trace formula. Lecture Notes in Mathematics, 1040.} 
Springer-Verlag, Berlin, 1983. i+128 pp.

\bibitem{gorodnik} A.Gorodnik, A.Nevo, G.Yehoshua, 
\emph{Counting lattice points in norm balls on higher rank simple Lie groups.} Math. Res. Lett. 24 (2017), 1285-1306.

\bibitem{gradry} I. S. Gradshteyn and I. M. Ryzhik, 
\emph{Table of integrals, series, and products. Translated from the Russian.} Translation edited and with a preface by Alan Jeffrey and Daniel Zwillinger. Seventh edition. Elsevier/Academic Press, Amsterdam, 2007. xlviii+1171 pp.



\bibitem{gunth}  P. G\"unther,
\emph{Gitterpunktprobleme in symmetrischen Riemannschen Räumen vom Rang 1.} 
Math. Nachr. 94, 5--27, 1980.

\bibitem{huber} H. Huber,
\emph{Zur analytischen Theorie hyperbolischen Raumformen und Bewegungsgruppen.} 
Math. Ann. 138 1959 1–26. 

\bibitem{hillparn} R. Hill and L. Parnovski,
\emph{The variance of the hyperbolic lattice point counting function.} 
Russ. J. Math. Phys. 12, no. 4, 472--482, 2005. 


\bibitem{iwaniec 2}H. Iwaniec,
\emph{ Prime geodesic theorem.} J. Reine Angew. Math., {\bf 349} (1984), 136--159.


\bibitem{iwaniec} H. Iwaniec,
\emph{Spectral methods of automorphic forms.} 
Second edition. Graduate Studies in Mathematics, 53. American Mathematical Society, Providence, RI; Revista Matemática Iberoamericana, Madrid, 2002. xii+220 pp. 

\bibitem{jarnik} V. Jarn\'ik, 
\emph{Uber die Mittelwerts\"atze der Gitterpunktlehre. V.} \v Casopis P\v est. Mat. Fys. {\bf 69} (1940), 148--174;

\bibitem{kaneko 2}
I.~Kaneko,
\emph{The prime geodesic theorem for $\mathrm{PSL}_2(\mathbb{Z}[i])$ and spectral exponential sums.}
 Algebra Number Theory \textbf{16} (2022), no.~8, 1845--1887.

\bibitem{kaneko}
I.~Kaneko,
\emph{The prime geodesic theorem for the Picard orbifold.}
preprint (2024)
\url{https://doi.org/10.48550/arXiv.2403.06626}.

\bibitem{kelmer} D. Kelmer,
\emph{On the mean square of the remainder for the euclidean lattice point counting problem.} Israel J. Math. {\bf 221} (2017), no.~2, 637--659

\bibitem{koyama1998} S.-y.~Koyama, \emph{Prime geodesic theorem for arithmetic compact surfaces.} Internat. Math. Res. Notices {\bf 1998}, no.~8, 383--388.

\bibitem{koyama}
S.-y.~Koyama, 
\emph{Prime geodesic theorem for the Picard manifold under the mean-Lindel\"of hypothesis.}
Forum Math. {\bf 13} (2001), no.~6, 781--793.



\bibitem{laaksonen} 
N.~Laaksonen,
\emph{Quantum Limits, Counting and Landau-type Formulae in Hyperbolic Space.}
PhD. Thesis. UCL 2016. \url{https://discovery.ucl.ac.uk/id/eprint/1474273/}.

\bibitem{laxphillips} P. Lax and R. Phillips,
\emph{The asymptotic distribution of lattice points in Euclidean and non-Euclidean spaces.} {\it Toeplitz centennial (Tel Aviv, 1981)}, pp. 365--375, Oper. Theory Adv. Appl., 4, Birkh\"auser Verlag, Basel-Boston, Mass.

\bibitem{levitan} B. M. Levitan,
\emph{Asymptotic formulas for the number of lattice points in Euclidean and Lobachevskiĭ spaces.} 
Uspekhi Mat. Nauk 42, no. 3 (255), 13--38, 255, 1987.



\bibitem{linderstrauss} E. Linderstrauss,
\emph{Invariant measures and arithmetic quantum unique ergodicity.} Ann. of Math. (2) {\bf 163} (2006), no.~1, 165--219.

\bibitem{luo sarnak} W. Luo and P.~C. Sarnak,
\emph{Quantum ergodicity of eigenfunctions on $PSL_2(\mathbb{Z})$.}
Inst. Hautes \'Etudes Sci. Publ. Math. No. 81 (1995), 207--237.

\bibitem{luo sarnak 2}W. Luo and P.~C. Sarnak,
\emph{Quantum variance for Hecke eigenforms.} Ann. Sci. \'Ecole Norm. Sup. (4) {\bf 37} (2004), no.~5, 769--799.



\bibitem{nelson} 
P.~Nelson,
\emph{Quantum variance on quaternion algebras I.}
preprint (2016). \url{https://doi.org/10.48550/arXiv.1601.02526}

\bibitem{parnovski} L. B. Parnovski, 
\emph{The Selberg trace formula and the Selberg zeta function for cocompact discrete subgroups $\so$.} 
Funct. Anal. Appl. {\bf 26} (1992), no.~3, 196--203; translated from Funktsional. Anal. i Prilozhen. {\bf 26} (1992), no.~3, 55--64.

\bibitem{patterson} S. J. Patterson,
\emph{A lattice-point problem in hyperbolic space.} 
Mathematika 22, no. 1, 81--88, 1975. 


\bibitem{petridisrisager}
Y.~Petridis and M.~Risager,
\emph{Local average in hyperbolic lattice point counting, with an appendix by Niko Laaksonen.}
Math. Z. \textbf{285} (2017), no.~3-4, 1319--1344.

\bibitem{phirud} R. Phillips and Z. Rudnick. 
\emph{The circle problem in the hyperbolic plane.} 
J. Funct. Anal. 121, no. 1, 78--116, 1994. 

\bibitem{rudnicksarnak} Z. Rudnick and P. C. Sarnak,
\emph{The behaviour of eigenstates of arithmetic hyperbolic manifolds.} Comm. Math. Phys.161(1994), no.1, 195–-213.

\bibitem{sarnak} P. C. Sarnak,
\emph{On cusp forms.} The Selberg trace formula and related topics (Brunswick, Maine, 1984), 393–407.
Contemp. Math., 53
American Mathematical Society, Providence, RI, 1986

\bibitem{selberg} A. Selberg,
\emph{Equidistribution in discrete groups and the spectral theory of automorphic forms.} \url{http://publications.ias.edu/selberg/section/2491}.

\bibitem{soundararajan} K. Soundararajan, 
\emph{Quantum unique ergodicity for $SL_2(\mathbb{Z})\setminus\H$.} Ann. of Math. (2) {\bf 172} (2010), no.~2, 1529--1538.


\bibitem{Zhao}
P.~Zhao,
\emph{Quantum variance of Maass--Hecke cusp forms.}
Comm. Math. Phys. {\bf 297} (2010), no.~2, 475--514.

\bibitem{Sarnak Zhao}
P.~Sarnak and P.~Zhao,
\emph{The quantum variance of the modular surface.}
Ann. Sci. \'Ec. Norm. Sup\'er. (4) {\bf 52} (2019), no.~5, 1155--1200.



\bibitem{zelditch1} 
S.~Zelditch,
\emph{Uniform distribution of eigenfunctions on compact hyperbolic surfaces.}
Duke Math. J. 55, (1987), 919--941.

\bibitem{zelditch2}
S.~Zelditch,
\emph{On the rate of quantum ergodicity I: Upper bounds.}
Comm. Math. Phys. \textbf{160} (1994), no.~1, 81--92. 


\bibitem{zelditch3}
S.~Zelditch,
\emph{Recent developments in mathematical quantum chaos.}
Current developments in mathematics, 2009, Int. Press, Somerville, MA, 2010, 115--204.

\end{thebibliography}
\end{document}